\documentclass[reqno,a4paper,10pt]{article}
\usepackage[T1]{fontenc}
\usepackage{mathptmx}
\usepackage[utf8]{inputenc}
\usepackage{amsmath}
\usepackage{graphicx}
\usepackage{amssymb}
\usepackage{amsthm}
\usepackage{color}
\usepackage{relsize}
\usepackage{graphicx}         
\usepackage{color}
\usepackage{multirow}
\usepackage{multicol}
\usepackage{subfigure}

\newtheorem{theorem}{Theorem}[section]
\newtheorem{lemma}[theorem]{Lemma}

\newtheorem{definition}[theorem]{Definition}

\theoremstyle{definition}
\begin{document}
%%%%%%%%%%%%%%%%%%%%%%%%%%%%%%%%%%%%%%%%%%%%%%%%%%%%%%%%%%%%%%%%%%%%%%%%%%%%%%%%%%%%%%%%%%%%%%%%%%%%%%%%%%%%%%%%%%%%%%%%%%%%%

\title{\textbf{{Moments of inverses of $(m,n,\beta)$-Laguerre matrices}}}
\author{Sushma Kumari  \\ { \small Department of Mathematics, Kyoto University} \\ {\footnotesize kumari93@math.kyoto-u.ac.jp}}
\date{}
\maketitle
%%%%%%%%%%%%%%%%%%%%%%%%%%%%%%%%%%%%%%%%%%%%%%%%%%%%%%%%%%%%%%%%%%%%%%%%%%%%%%%%%%%%%%%%%%%%%%%%%%%%%%%%%%%%%%%%%%%%%%%%%%%%%
\begin{abstract}
Wishart matrices are one of the fundamental matrix models in multivariate statistics. We consider the classical $(m,n,\beta)$-Laguerre ensemble and give a necessary and sufficient condition for finite moments for the inverse of $(m,n,\beta)$-Laguerre matrices to exist. We extend the result to inverse compound Wishart matrices for the values of $\beta = 1$ and $2$. Our result complements the result by Letac and Massam \cite{Letac_Massam_2004}, Matsumoto \cite{Matsumoto_2012} and Collins et al. \cite{Collins_Matsumoto_Saad_2014}.

\end{abstract}
%%%%%%%%%%%%%%%%%%%%%%%%%%%%%%%%%%%%%%%%%%%%%%%%%%%%%%%%%%%%%%%%%%%%%%%%%%%%%%%%%%%%%%%%%%%%%%%%%%%%%%%%%%%%%%%%%%%%%%%%%%%%%
%%%%%%%%%%%%%%%%%%%%%%%%%%%%%%%%%%%%%%%%%%%%%%%%%%%%%%%%%%%%%%%%%%%%%%%%%%%%%%%%%%%%%%%%%%%%%%%%%%%%%%%%%%%%%%%%%%%%%%%%%%%%
\section{Introduction}
\label{sec:introduction}
Real Wishart matrices were introduced by Wishart \cite{Wishart_1928} in 1928, while complex Wishart matrices were studied by Goodman in \cite{Goodman_1963}. Initially, the classical Wishart ensembles were studied only for the values of $\beta = 1,2$ and $4$ corresponding to real, complex and quaternion Wishart matrices respectively. Then, Dumitriu and Edelman \cite{Dumitriu_Edelman_2002}  generalised the classical Wishart models to tridiagonal matrix models for the general values of $\beta >0$. Another generalisation of the Wishart matrices, called compound Wishart matrices for the values $\beta = 1,2$ and $4$, was firstly studied by Speicher in \cite{Speicher_1998}.

Many properties of above random matrices such as the eigenvalues densities and moments play a very important role in various fields of mathematics and physics. The moments of Wishart and inverse Wishart matrices have been studied rigourously and the explicit formulas were given in \cite{Graczyk_Letac_Massam_2003} and \cite{Matsumoto_2012}. 

Letac and Massam \cite{Letac_Massam_2004} were the first one to compute all the general moments of Wishart and inverse Wishart matrices of the form  $\mathbb{E}(Q(S))$ and $\mathbb{E}(Q(S^{-1}))$ in both real and complex cases, where $Q$ is a polynomial depending only on the eigenvalues of these matrices. Later Matsumoto \cite{Matsumoto_2012}, computed all the general moments of Wishart and inverse Wishart matrices using Weingarten function. The expression for the moments of inverse compound Wishart matrices was also obtained by Collins et al. in \cite{Collins_Matsumoto_Saad_2014}. 

In \cite{Graczyk_Letac_Massam_2003}, \cite{Letac_Massam_2004}, \cite{Matsumoto_2012} and \cite{Collins_Matsumoto_Saad_2014}, to compute the c-th moment of inverse of Wishart matrices, a sufficient condition such as $c < m-n+1$ for $\beta =2 $ (in complex case) or $c < (m-n+1)/{2}$  for $\beta =1$ (in real case) was assumed. Interestingly, it is not known whether this condition is necessary or not to have finite moments. This is the motivation for our work (refer to section \ref{sec:motivation}). In general, we consider the question for the broader class of inverse $(m,n,\beta)$-Laguerre matrices. 
 
In this paper, we present a necessary and sufficient condition to obtain finite moments for the inverse of $(m,n,\beta)$-Laguerre matrices and compound Wishart matrices. We can formulate our problem as follows: given a $(m,n,\beta)$-Laguerre matrix $S$, find a linear function $\displaystyle \ {g(m,n,\beta)}$ such that 
\begin{center}
$\ { \displaystyle \mathbb{E}\{ {\rm Tr}(S^{-c})\} < \infty \text{ \ \text{ if and only if} \ } c < g(m,n,\beta) }$
\end{center} 
$\ {\displaystyle \mathbb{E}\{ {\rm Tr}(S^{-c})\} }$ is the $c$-th moment of the matrix $S^{-1}$.

Our approach is very elementary in nature and we extend the result to inverse compound Wishart matrices. This paper is arranged as follows : in section 2, we give some basic definitions, notations and lemmas. In section 3, we present our result and further some conclusions in section 4.

Our main results can be phrased as follows :
\begin{theorem}
\label{thm:beta_moment}
Let $S$ be a $n \times n$ $(m,n,\beta)$-Laguerre matrix for $\beta > 0$. Then for any integer $c > 0$, 
\begin{align*}
\displaystyle \mathbb{E}\{{ \rm Tr }(S^{-c})\} \text{\  is finite if and only if \ } c < (m-n+1)\beta/{2} 
\end{align*}
\end{theorem}

\begin{theorem}
\label{thm:compound_moment}
Let $Q$  a $n \times n$ non-degenerate compound Wishart matrix for the values $\beta = 1$ or $2$. For any integer $c > 0$, we have
\begin{align*}
\mathbb{E}\{{ \rm Tr}(Q^{-c})\} \text{\  is finite if and only if \ } c < (m-n+1)\beta /2
\end{align*}
\end{theorem}

The proof of Theorem \ref{thm:beta_moment} and \ref{thm:compound_moment} is given in section \ref{sec:result}.
%%%%%%%%%%%%%%%%%%%%%%%%%%%%%%%%%%%%%%%%%%%%%%%%%%%%%%%%%%%%%%%%%%%%%%%%%%%%%%%%%%%%%%%%%%%%%%%%%%%%%%%%%%%%%%%%%%%%%%%%%%%%
%%%%%%%%%%%%%%%%%%%%%%%%%%%%%%%%%%%%%%%%%%%%%%%%%%%%%%%%%%%%%%%%%%%%%%%%%%%%%%%%%%%%%%%%%%%%%%%%%%%%%%%%%%%%%%%%%%%%%%%%%%%%
\section{Definitions and Notations}
\label{sec:Def_and_not}

Let $M_{m,n}$ denote the space of $m\times n$ matrices with entries from $\mathbb{R}$ or $\mathbb{C}$ depending on the context. $M_{n}$ is the space of $n\times n$ matrices. We write ${\rm Tr}(S)$ for the un-normalised trace of a matrix $S$.

\begin{definition}[Loewner Partial order ]
Let $A,B \in M_m$. We write $A \preceq B$ if $A$ and $B$ are Hermitian matrices and $B-A$ is a positive semidefinite. The order relation $"\preceq "$ is referred as Loewner partial order.
\end{definition}

\begin{lemma}(See theorem 7.7.2 in \cite{Horn_Johnson_1986})
Let $A,B \in M_m$ be two Hermitian matrices. Let $\sigma_1(A) \leq \sigma_2(A) \leq \dots \leq \sigma_m(A)$ and $\sigma_1(B) \leq \sigma_2(B) \leq \dots \leq \sigma_m(B)$ be the ordered eigenvalues of $A$ and $B$, respectively. If $A \preceq B$, then
\begin{enumerate}
\renewcommand{\labelenumi}{(\roman{enumi})}
\item $S^*AS \preceq S^*BS$ for  $S \in M_{m,n}$.
\item  $\sigma_i(A) \leq \sigma_{i}(B) $ for each $i = 1,\dots,m$.
\end{enumerate}
\end{lemma}

\begin{lemma}\cite{Ross_2006}
\label{positive_rv}
Let $Z$ be a positive random variable and let $g(z)$ be a measurable function of $z$. If $a$ is any real such that $\mathbb{P}(Z \geq a) = 1$, then
\begin{align*}
\displaystyle \mathbb{E}\{g(Z)\} = g(a) + \int_{a}^{\infty} g'(z) \mathbb{P}(Z>z) dz 
\end{align*}
\end{lemma}

Let $K_{m,n}(\mathcal{F})$ be the collection of $m\times n$ random matrices with independent and identically distributed entries  following the distribution $\mathcal{F}$. Let $\mathcal{N}{(0,1)}$ and $\tilde{\mathcal{N}}(0,1)$ refer to the standard real and  complex Gaussian distribution respectively. $\tilde{\mathcal{N}}(0,1)$ is the standard Gaussian distribution of $(x+iy)/ \sqrt{2}$ where $x,y \sim {\mathcal{N}}\left(0,1 \right)$ and $x,y$ are independent. 

Let $A \in K_{m,n}(\tilde{\mathcal{N}}(0,1))$, then the $n \times n$ matrix of the type $P = A^{*}A$ is called a standard \emph{complex Wishart} matrix. Similarly, we say $P$ is a standard \emph{real Wishart} matrix if $ A \in K_{m,n}({\mathcal{N}}(0,1) )$.

\begin{definition}[$(m,n,\beta)$-Laguerre matrix](Dumitriu and Edelman \cite{Dumitriu_Edelman_2002})
Let $X$ be a bidiagonal matrix with mutually independent diagonal and subdiagonal entries with the following distribution
   \begin{align*}
	\displaystyle X \sim 
	\begin{pmatrix}
	\chi_{ m \beta} & \chi_{(n-1)\beta} & & & \\
	& \chi_{(m-1)\beta}& \chi_{(n-2)\beta}& & \\
	& \ddots & \ddots & & \\
	& & \chi_{(m-n+2)\beta} & \chi_{\beta}& \\
	& &  & \chi_{(m-n+1) \beta}&
	\end{pmatrix}
	\end{align*}
The tridiagonal matrix $S = X^*X$ is called a \emph{$(m,n,\beta)$-Laguerre matrix} and $S^{-1}$ denote the inverse of $(m,n,\beta)$-Laguerre matrix $S$. Here, $\chi_{s}$ is the chi distribution with parameter $s$.
\end{definition}
For the values $\beta = 1 \text{\ and \ } 2$, the eigenvalue distribution of a standard real and complex Wishart matrix is same as the eigenvalue distribution of a $(m,n,1)$-Laguerre and  $(m,n,2)$-Laguerre matrix respectively.
 
Let $m \geq n$. The matrix $S$ has $n$ real eigenvalues whose joint probability density function is stated below.

\begin{theorem}[Joint eigenvalue density \cite{Dumitriu_Edelman_2002}]
\label{jpdf_eigenvalues} 
Let $0 < \lambda_{1} \leq \lambda_2 \leq \dots \leq \lambda_{n}$ be the $n$ eigenvalues of $(m,n,\beta)$-Laguerre matrix $S$ for $\beta > 0$. The joint eigenvalue density function is given as : 
\begin{align}
\label{eqn:jpdf_eigenvalues} 
\displaystyle h_{\beta}(\lambda_1,\lambda_2,\dots,\lambda_n) \  = \  Z_{m,n}^{\beta} \prod_{i=1}^{n}\lambda_{i}^{\alpha -1} e^{\left(- \frac{1}{2}\sum_{i=1}^{n}\lambda_i\right)} \prod_{k<j}(\lambda_{j} - \lambda_{k})^{\beta}
\end{align}
where $\alpha = (m-n+1)\beta/2$. $Z_{m,n}^{\beta} $ is a normalisation constant that can be computed explicitly:
\begin{align*}
\displaystyle  Z_{m,n}^{\beta} \ = \  2^{-mn\beta/2} \prod_{j=1}^{n}  \frac{ \Gamma{\left(1+\frac{\beta}{2}\right)} }{\Gamma{\left(1+\frac{\beta }{2}j\right)} \Gamma{\left( \frac{\beta}{2}(m-n+j)\right)}}
\end{align*}
\end{theorem}

\begin{definition}
\label{def:compound_wishart}
\begin{enumerate}
\renewcommand{\labelenumi}{(\roman{enumi})}
\item Let $\Sigma$ be a $n \times n$ positive definite Hermitian matrix and $B$ be a $m \times m$ complex matrix. Let $X $ be the $m \times n$ standard complex Wishart matrix. Then,  
\begin{align*}
\displaystyle Q \ = \  \sqrt{\Sigma}X^{*}BX\sqrt{\Sigma}
\end{align*}
is called a complex $\emph{compound Wishart}$ matrix.

\item 
If $B$ is a positive-definite Hermitian matrix then the matrix $B$ has an eigenvalue decomposition i.e. $B  = UDU^*$, where $U $ is an unitary matrix consisting of eigenvectors of $B$ and $D = diag({\xi_1,\xi_2,\dots,\xi_m}) \text{\ such that \ } 0< \xi_1 \leq \xi_2 \leq \dots \leq \xi_m$. So, $Q$ has the same distribution as $X^{*}DX$.
\end{enumerate}

\end{definition}

The real compound Wishart matrices can be defined analogously. 
%%%%%%%%%%%%%%%%%%%%%%%%%%%%%%%%%%%%%%%%%%%%%%%%%%%%%%%%%%%%%%%%%%%%%%%%%%%%%%%%%%%%%%%%%%%%%%%%%%%%%%%%%%%%%%%%%%%%%%%%%%%%%
\subsection{Motivation}
\label{sec:motivation}
Let us state the moment formula for the inverse of a complex Wishart matrix as given by Graczyk et. al. in \cite{Graczyk_Letac_Massam_2003}.

Let $\mathcal{S}_q$ be the symmetric group defined on $[q] \ = \ \{1,2,\dots,q\}$ for the positive integer $q$. Every permutation $\sigma \in \mathcal{S}_q$ can be decomposed into cycles with length $(\eta_1,\eta_2,\dots,\eta_l)$. If $\eta_1 \geq \eta_2 \geq \dots \geq \eta_l$ then $\eta = (\eta_1,\eta_2,\dots,\eta_l)$ is the partition of $q$. Let $p(\eta)$ be the length of $\eta$. For a matrix $A$ and $\sigma \in \mathcal{S}_{q}$, define	
\begin{align*}
{\rm Tr }_{\sigma}(A) \ = \ \prod_{i=1}^{l} { \rm Tr}(A^{\eta_i}) 
\end{align*}
Let $\sigma \in \mathcal{S}_q $ and $z \in \mathbb{C}$. Then the unitary Weingarten function can be defined as
\begin{align*}
\displaystyle {\rm Wg}(\sigma,z) \ = \ \frac{1}{q!} \ {\underset{\eta}{\mathlarger{\sum}}} \  \frac{\chi^{\eta}(e)}{ \prod_{i=1}^{p(\eta)} \prod_{j=1}^{\eta_i} (z + j - i)} \chi^{\eta}(\sigma)
\end{align*}
where the $\chi^{\lambda}$ are irreducible characters in $\mathcal{S}_q$.

\begin{theorem}[\cite{Graczyk_Letac_Massam_2003}]
\label{thm:moment_letac}
Let $S$ be a $n \times n$ complex Wishart matrix and  $\pi \in \mathcal{S}_{c}$. If $c < (m -n + 1)$,
\begin{align*}
\displaystyle \mathbb{E}\{ { \rm Tr}_{\pi}(S^{-1}) \} = \ (-1)^{c} \sum_{\sigma \in \mathcal{S}_{c}} {\rm Wg}(\pi {\sigma}^{-1};n-m)Tr_{\sigma}(I)
\end{align*}
where $I$ is the $n \times n$ identity matrix and ${ \rm Wg}(\pi \sigma^{-1};n-m) $ is the unitary Weingarten function.
\end{theorem}

From theorem $\ref{thm:moment_letac}$, it can be seen that $c < (m-n+1)$ is assumed to define the moments of the inverse  of a complex Wishart matrix. Similar condition has been assumed in $\cite{Letac_Massam_2004}$, \cite{Matsumoto_2012} and \cite{Collins_Matsumoto_Saad_2014} while working with the moments of the inverses of real and complex Wishart matrices. Our aim is to investigate the necessity of this condition. 
%%%%%%%%%%%%%%%%%%%%%%%%%%%%%%%%%%%%%%%%%%%%%%%%%%%%%%%%%%%%%%%%%%%%%%%%%%%%%%%%%%%%%%%%%%%%%%%%%%%%%%%%%%%%%%%%%%%%%%%%%%%%
%%%%%%%%%%%%%%%%%%%%%%%%%%%%%%%%%%%%%%%%%%%%%%%%%%%%%%%%%%%%%%%%%%%%%%%%%%%%%%%%%%%%%%%%%%%%%%%%%%%%%%%%%%%%%%%%%%%%%%%%%%%%
\section{Results}
\label{sec:result}

Let  $ 0 < \lambda_{1} \leq \dots \leq \lambda_n$ be the ordered eigenvalues of $(m,n,\beta)$-Laguerre matrix $S$, so $\lambda_1$ is the smallest eigenvalue of $S$. We first present a result on the gap probability of smallest eigenvalue of $S$ near zero. The term \emph{\lq gap probability at zero'} means the probability that no eigenvalue of the matrix lies in the  neighbourhood of zero. Let $a>0$,
\begin{align*}
\displaystyle \mathbb{P}(\text{no eigenvalues} \in (0,a)) \ = \ & \ \mathbb{P}(\lambda_1 \notin (0,a)) \ \   \ \\
\displaystyle 	\ = \ & \   1 -  \mathbb{P}(\lambda_1 < a)  \ 
\end{align*}

\begin{lemma}
\label{lem:gap_prob}
Let $S$ be a $(m,n,\beta)$-Laguerre matrix, 
\begin{align}
\displaystyle	\mathbb{P}(\lambda_1 < a) \ \underset{a\rightarrow 0}{\approx} \ C^{\beta}_{m,n} a^{\alpha}  \ \ 
\end{align}
where the constant $C^{\beta}_{m,n}$ is non-zero and depends only on $m,n$ and $\beta$ and $\alpha = (m-n+1)\beta/2$.
\end{lemma}
\begin{proof}
From \eqref{eqn:jpdf_eigenvalues}, we have
\begin{align*}
\displaystyle   \mathbb{P}(\lambda_1 < a) \ = \ & \ Z_{m,n}^{\beta} \int_{0}^{a} \int_{\lambda_1}^{\infty} \dots \int_{\lambda_{n-1}}^{\infty}\prod_{i=1}^{n}\left(\lambda_{i}^{\alpha -1} \right)\prod_{k<j}(\lambda_{j} - \lambda_{k})^{\beta} e^{\left(-\frac{1}{2}\sum_{i=1}^{n}\lambda_i\right)} d\lambda_{n} \dots d\lambda_1 
\end{align*}
Using change of variable, $(\lambda_1,\lambda_2,\dots,\lambda_n) = (a x_1,a x_1 + x_2,\dots,a x_1 + x_2 + \dots x_n)$,
\begin{align*}
\displaystyle \mathbb{P}(\lambda_1 < a) \ = \ & \ Z_{m,n}^{\beta} a^{\alpha}  \int_{0}^{1} \int_{0}^{\infty} \dots \int_{0}^{\infty}
f_{a}(x_1,x_2,\dots,x_n)dx_{n} \dots dx_1 
\end{align*}
where, 
\begin{align*}
f_{a}(x_1,x_2,\dots,x_n) \ = \ x_{1}^{\alpha-1}  e^{(-nx_1 a/2)} \prod_{i=2}^{n}\left(ax_1 + \sum_{j=2}^{i} x_j\right)^{\alpha-1}  \prod_{j=2}^{n} x_{j}^{\beta} \prod_{j=2}^{n} e^{-(n-j+1)x_j /2} \\ \prod_{k=3}^{n}\left( \prod_{i=1}^{k-2} \left( \sum_{j=i+1}^{k} x_j \right)^{\beta} \right)
\end{align*}
It follows,
\begin{align*}
\lim_{a \rightarrow 0} \  f_{a}(x_1,\dots,x_n) \  = \ &  \ f(x_1,\dots,x_n) \ \\
 \ = \ & \ x_{1}^{\alpha-1}\prod_{i=2}^{n}\left(\sum_{j=2}^{i} x_j\right)^{\alpha-1} \prod_{k=3}^{n}\left( \prod_{i=1}^{k-2} \left( \sum_{j=i+1}^{k} x_j \right)^{\beta} \right) \prod_{j=2}^{n} x_{j}^{\beta}e^{-(n-j+1)x_j / 2} 
\end{align*}
We next find a dominating function for $f_a(x_1,\dots,x_n)$. By some straightforward calculations, we have
\begin{enumerate}
\renewcommand{\labelenumi}{(\roman{enumi})}
\item $ 
\!
\begin{aligned}[t]
  \prod_{i=2}^{n}\left(ax_1 + \sum_{j=2}^{i} x_j\right)^{\alpha-1} \  \leq \ & \ \prod_{j=2}^{n} {x_j}^{-1} \ \prod_{i=2}^{n}\left(1 + \sum_{j=2}^{i} x_j\right)^{\alpha}  \ \\
 \ \leq  \  & \ \prod_{j=2}^{n}{x_j}^{-1} \left(1 + \sum_{j=2}^{n} x_j\right)^{(n-1)\alpha}  
\end{aligned}
$
\item $ 
\!
\begin{aligned}[t]
\prod_{k=3}^{n}\left( \prod_{i=1}^{k-2} \left( \sum_{j=i+1}^{k} x_j \right)^{\beta} \right) \ \leq \ & \ \prod_{k=3}^{n}\left( \prod_{i=1}^{k-2} \left( 1 + \sum_{j=2}^{n} x_j \right)^{\beta} \right) \ \\
\ \leq \ & \ \left( 1 + \sum_{j=2}^{n} x_j^{(n-1)(n-2)\beta /2} \right)
\end{aligned}
 $
\end{enumerate}
Therefore, we have 
\begin{align*}
|f_{a}(x_1,\dots,x_n) | \ \leq \ \left|x_1^{\alpha -1} \right|  \left| \prod_{j=2}^{n} {x_j}^{\beta -1}e^{-(n-j+1)x_j/2}  \left( 1 + \sum_{j=2}^{n} x_j^{(n-1)(n-2)\beta /2 + (n-1)\alpha} \right) \right|
\end{align*}
Using the inequality, for any positive real $x,y$ and $p >0$, 
\begin{align*}
(x + y)^{p} \ \leq \ 2^{p} (x^{p} + y^{p})  
\end{align*}
and put $p \ = \ (n-1)\alpha + (n-1)(n-2)\beta /2 > 0$, 
\begin{align*}
g(x_1,\dots,x_n) \  = \ & \ {x_1}^{\alpha-1} \ \prod_{j=2}^{n} {x_j}^{\beta-1} e^{-(n-j+1)x_j /2} \left(2^{p} (1+x_2)^{p}  + \sum_{j=3}^{n} 2^{(j-2)p} {x_j}^{p} \right) 
\end{align*}
Note that, $g$ is a finite sum of integrable function and hence is integrable and, 
$$ |f_{a}(x_1,\dots,x_n) | \ \leq \ \left| g{(x_1,\dots,x_n)} \right| $$
By dominated convergence theorem, in the limit of $a \rightarrow 0$
\begin{align*}
\int_{0}^{1} \int_{0}^{\infty} \dots \int_{0}^{\infty}
f_{a}(x_1,x_2,\dots,x_n)dx_{n}\dots dx_1  \ \approx \ & \ \int_{0}^{1} \int_{0}^{\infty} \dots \int_{0}^{\infty}
f(x_1,x_2,\dots,x_n)dx_{n}\dots dx_1 \ \\
\ = \ &  \ z_{m,n}^{\beta}\ \  \text{ ( $> \ 0$  by the positivity of $f$  ) } \ \\
 \text{Hence, } \  \hspace{5cm}  \\
\mathbb{P}(\lambda_1  <  a) \ \underset{a\rightarrow 0}{\approx} \ & \  a^{\alpha} Z_{m,n}^{\beta} z_{m,n}^{\beta}  \ \\
\ = \ & \ a^{\alpha} C^{\beta}_{m,n}  
\end{align*}
\end{proof}
\subsection{Inverse $(m,n,\beta)$-Laguerre matrix}
We now present the proof of theorem $\ref{thm:beta_moment}$.

\begin{proof}[\textbf{Proof of Theorem \ref{thm:beta_moment}.}]
For any integer $c > 0$, we have
\begin{align}
\displaystyle {\lambda_{1}^{-c}} \ & \ \leq \ \sum_{i=1}^{n} {\lambda_{i}^{-c}} \  \leq  \  n{\lambda_{1}^{-c}}  \nonumber    \\
 \displaystyle   \hspace{2cm} \mathbb{E} \left\{ {\lambda_{1}^{-c}} \right\} \ & \ \leq  \  \mathbb{E} \left\{ Tr(S^{-c}) \right\}  \  \leq  \ n \mathbb{E}\left\{{\lambda_{1}^{-c}} \right\} 
\label{eqn:trace_ev}
\end{align}

$ \ \mathbb{E}\left\{ Tr(S^{-c} \right\} < \infty \text{ if and only if \ }  \displaystyle \mathbb{E} \left\{ {\lambda_{1}^{-c}} \right\} < \infty $ and thus, it is sufficient to find the finiteness condition for $\mathbb{E} \left\{ {\lambda_{1}^{-c}} \right\}$. 

As $\lambda_{1}^{-1}$ is a positive random variable, from lemma \ref{positive_rv}
\begin{align*}
\displaystyle  \mathbb{E} \left\{ {\lambda_{1}^{-c}} \right\} \ = \ & \ \int_{0}^{\infty} c t^{c-1} \mathbb{P}\left({\lambda_1^{-1}} > t \right) dt \ \\
\displaystyle \ = \ & \ \displaystyle  c \  \int_{0}^{\infty} t^{c-1} \mathbb{P}\left(\lambda_1 <	 {t^{-1}} \right)  dt \ \\
\displaystyle \ = \ & \ \displaystyle  c \ \int_{0}^{\infty}  w^{-c-1} \mathbb{P}\left(\lambda_1 <	w \right)  dw     \text{\ \ \ \ \ \ \ \ (put $1/t = w$)} \ \\
\displaystyle \ = \ & \ \displaystyle   c \  \int_{0}^{\delta} w^{-c-1} \mathbb{P}\left(\lambda_1 < w \right) dw +   c \ \int_{\delta}^{\infty}  w^{-c-1} \mathbb{P}\left(\lambda_1 < w \right) dw
\end{align*}
As a result,
\label{eqn:bound}
\begin{enumerate}
\renewcommand{\labelenumi}{(\roman{enumi})}
\item 
$\displaystyle   c \ \int_{0}^{\delta} w^{-c-1} \mathbb{P}\left(\lambda_1 < w \right) dw \ \leq \ \mathbb{E} \left\{ {\lambda_{1}^{-c}} \right\}  $  and, 
\item
$ \displaystyle \mathbb{E} \left\{ {\lambda_{1}^{-c}} \right\} \  \leq \   c \ \int_{0}^{\delta} w^{-c-1} \mathbb{P}\left(\lambda_1 < w \right) dw \ + \   c \ \int_{\delta}^{\infty}  w^{-c-1} dw  $
\end{enumerate}

Consider the inequality (i), 
\begin{align*}
\displaystyle  \mathbb{E} \left\{ {\lambda_{1}^{-c}} \right\} \ \geq \ & \  c \int_{0}^{\delta}  w^{-c-1} \mathbb{P}\left(\lambda_1 < {w} \right) \ dw \ \\
\displaystyle  \approx \ & \ c \  C^{\beta}_{m,n}\int_{0}^{\delta} w^{-c-1} w^{\alpha} \ dw  \\
\displaystyle  = \ & \  \infty \ \ \ \text { if \ } (\alpha - c) \leq 0 
\end{align*}

Evaluating the inequality (ii),
\begin{align*}
\displaystyle \mathbb{E} \left\{ {\lambda_{1}^{-c}} \right\} \ \leq \ & \  c \  \int_{0}^{\delta} w^{-c-1} \mathbb{P}\left(\lambda_1 < w \right) \ dw \ + \  c \ \int_{\delta}^{\infty}  w^{-c-1} \ dw \ \\
\displaystyle  \approx \ & \  c \ C^{\beta}_{m,n} \  \int_{0}^{\delta}  w^{-c-1} w^{\alpha} dw \ + \   c \ \int_{\delta}^{\infty}  w^{-c-1}  dw \ \\
\displaystyle \ < \ & \  \infty \ \ \ \text { if \ } {(\alpha - c )} > 0
\end{align*}
This implies that, $\displaystyle  \hspace{.1cm} \mathbb{E} \left\{ {\lambda_{1}^{-c}} \right\} < \infty$ if and only if $c < \alpha$ and hence, $\mathbb{E}\left\{ Tr(S^{-c}) \right\}$ is finite if and only if $c < (m-n+1)\beta/2$.
\end{proof}

%%%%%%%%%%%%%%%%%%%%%%%%%%%%%%%%%%%%%%%%%%%%%%%%%%%%%
\subsection{Inverse compound Wishart matrix}
\label{sec:complex_compound_wishart}
Let $Q$ be a $n \times n$ compound Wishart matrix. As, $Q$ has the same distribution as $X^*DX$, using the partial order we get, $\xi_1I \preceq D \preceq \xi_mI$, where $I$ is a  $m \times m$ identity matrix and thus, $\xi_1S \preceq Q \preceq \xi_m S$. Let $0 < \mu_1\leq \mu_2\leq \dots\leq \mu_n$ be the ordered eigenvalues of $Q$ then for any real $a >0$,
\begin{align}
\displaystyle \xi_1\lambda_1 \ \leq \  \ \mu_1 \leq \ & \ \xi_m \lambda_1 \nonumber \ \\
\displaystyle \mathbb{P}(\lambda_1 < a{\xi_1}^{-1}) \ \leq \ \mathbb{P}(\mu_1 < a) \leq \ &  \mathbb{P}(\lambda_1 < a{\xi_m}^{-1}) 
\label{eqn:bound_mu} 
\end{align}
From the lemma \ref{lem:gap_prob}, as $a \rightarrow 0$ we get
\begin{align*}
\displaystyle \ \frac{C^{\beta}_{m,n}}{{\xi_1}^{\alpha}} \  a^{\alpha} \ \leq \ \mathbb{P}(\mu_1 < a) \ \leq \ \frac{C^{\beta}_{m,n} }{{\xi_m}^{\alpha}} \ a^{\alpha} \
\end{align*}
From \eqref{eqn:trace_ev}, it is enough to compute the finiteness condition for $\mathbb{E}\{ \mu_1^{-c}\} $ instead of $\mathbb{E}\{Tr( Q^{-c})\}$.

\begin{proof}[\textbf{Proof of Theorem \ref{thm:compound_moment}.}]
Proceeding as in the proof of theorem \ref{thm:beta_moment} and using \eqref{eqn:bound_mu}, 
\begin{align*}
\displaystyle \mathbb{E} \left\{ {\mu_{1}^{-c}} \right\} \ \leq  \  &  \ c \ \int_{0}^{\delta}  w^{-c-1} \mathbb{P}\left(\mu_1 < w \right) dw   \ +   \  c \ \int_{\delta}^{\infty}  w^{-c-1} \ dw \ \\
%%\ \displaystyle  \leq \ &  \ c \ \int_{0}^{\delta}  w^{-c-1} \mathbb{P}\left(\lambda_1 < w / {\xi_m} \right) \ dw  \ + \  c \ \int_{\delta}^{\infty}  w^{-c-1}  \ dw \ \\
\displaystyle \ \leq \ &  \ \frac{c \ C^{\beta}_{m,n}}{\xi_m^{\alpha}} \ \int_{0}^{\delta}  w^{\alpha-c-1}  dw   \ +   \  c \ \int_{\delta}^{\infty}  w^{-c-1}  \ dw \ \\
\displaystyle  \ < \ & \ \infty \  \ \ \ \text { if \ } {(\alpha - c )} > 0
\end{align*}
Computing the lower bound for $\mathbb{E} \left\{ {\mu_{1}^{-c}} \right\}\ $,
\begin{align*}
\displaystyle \mathbb{E} \left\{ {\mu_{1}^{-c}} \right\}\  \geq \ & \ c \ \int_{0}^{\delta}  w^{-c-1} \mathbb{P}\left(\mu_1 < w \right) \ dw \  \\
%%\displaystyle  \ \geq \  &  \ c \ \int_{0}^{\delta}  w^{-c-1} \mathbb{P}\left(\lambda_1 < w / {\xi_1} \right) dw  \ \\
\displaystyle  \ \geq \ & \  \ \frac{c \ C^{\beta}_{m,n}}{\xi_1^{\alpha}} \ \int_{0}^{\delta}  w^{\alpha-c-1}  dw  \ \\
\displaystyle  \ = \ &  \ \infty \  \ \ \ \text { if \ } {(\alpha - c )} \leq 0
\end{align*}

Hence, the result.
\end{proof}

%%%%%%%%%%%%%%%%%%%%%%%%%%%%%%%%%%%%%%%%%%%%%%%%%%%%%%%%%%%%%%%%%%%%%%%%%%%%%%%%%%%%%%%%%%%%%%%%%%%%%%%%%%%%%%%%%%%%%%%%%%%%%%%%
%%%%%%%%%%%%%%%%%%%%%%%%%%%%%%%%%%%%%%%%%%%%%%%%%%%%%%%%%%%%%%%%%%%%%%%%%%%%%%%%%%%%%%%%%%%%%%%%%%%%%%%%%%%%%%%%%%%%%%%%%%%%%%%%
\section{Conclusion}

The lemma \ref{lem:gap_prob} and theorem \ref{thm:beta_moment} give two important information: firstly 
it investigates the behaviour of the smallest eigenvalue of $(m,n,\beta)$-Laguerre matrices near zero and, secondly it state the necessary and sufficient conditon for the existence of finite moments for inverse $(m,n,\beta)$-Laguerre matrices. In simpler words, if we are provided the value of $m,n$ and $\beta$ then we can tell whether the $c$-th moment of the inverse of a $(m,n,\beta)$-Laguerre matrix is finite or not and the same observation follows for the moments of smallest eigenvalue of a $(m,n,\beta)$-Laguerre matrix. 
We can summarise our result for matrix $S$ as:

 \emph{ $\displaystyle \mathbb{E}\{Tr(S^{-c})\} < \infty $ if and only if $\displaystyle  c < (m-n+1)\beta /2 $ \ i.e. the finite integer moments of the inverse of a $(m,n,\beta)$-Laguerre matrix lies in the interval $ \left(0, (m-n+1)\beta/2 \right)$. }
		
	We study the compound Wishart matrix for the values $\beta =1$ and $2$. As, compound Wishart matrices are the natural generalisation to Wishart matrices, so the result can be extended easily to compound Wishart matrices. The $c$-th moment for the inverse of a compound Wishart matrix exist if and only if $c < (m-n+1)\beta /2 $ and thus, we obtain that the finite integer moments lies in $\displaystyle \left(0, (m-n+1)\beta/2 \right)$. 
	
	Recently, the negative moments of $(m,n,\beta)$-Laguerre matrices has been studied in \cite{Mezzadri_Reynolds_Winn_2017}. Our results are consistent and complete with the other results on the moments of $(m,n,\beta)$-Laguerre matrices and compound Wishart matrices as in \cite{Letac_Massam_2004}, \cite{Matsumoto_2012}, \cite{Mezzadri_Reynolds_Winn_2017} and \cite{Collins_Matsumoto_Saad_2014}. 
	In the general $\beta$ case, there is no well defined notion of compound Wishart matrix, which explains why we focused on compound Wishart case for the values of $\beta = 1$ and $2$. However we expect that our results can be extended to  more general matrix models involving Wishart matrices and leave it for the future work.  \ \\

%%%%%%%%%%%%%%%%%%%%%%%%%%%%%%%%%%%%%%%%%%%%%%%%%%%%%%%%%%%%%%%%%%%%%%%%%%%%%%%%%%%%%%%%%%%%%%%%%%%%%%%%%%%%%%%%%%%%%%%%%%%%%
%%%%%%%%%%%%%%%%%%%%%%%%%%%%%%%%%%%%%%%%%%%%%%%%%%%%%%%%%%%%%%%%%%%%%%%%%%%%%%%%%%%%%%%%%%%%%%%%%%%%%%%%%%%%%%%%%%%%%%%%%%%%%
\textbf{Acknowledgement}

I would like to express my sincere gratitude to Prof. Benoit Collins, my PhD supervisor, for his guidance. His constant encouragement and useful comments has helped me a lot. I am also grateful to Prof. Matsumoto for his insightful suggestions and comments on the preliminary version of this manuscript. I am thankful to all those researchers whose work has motivated me.

%%%%%%%%%%%%%%%%%%%%%%%%%%%%%%%%%%%%%%%%%%%%%%%%%%%%%%%%%%%%%%%%%%%%%%%%%%%%%%%%%%%%%%%%%%%%%%%%%%%%%%%%%%%%%%%%%%%%%%%%%%%%%
%%%%%%%%%%%%%%%%%%%%%%%%%%%%%%%%%%%%%%%%%%%%%%%%%%%%%%%%%%%%%%%%%%%%%%%%%%%%%%%%%%%%%%%%%%%%%%%%%%%%%%%%%%%%%%%%%%%%%%%%%%%%%
\bibliographystyle{IEEEtrans} 
\bibliography{Moments_of_beta_laguerre_matrices}
%%%%%%%%%%%%%%%%%%%%%%%%%%%%%%%%%%%%%%%%%%%%%%%%%%%%%%%%%%%%%%%%%%%%%%%%%%%%%%%%%%%%%%%%%%%%%%%%%%%%%%%%%%%%%%%%%%%%%%%%%%%%
%%%%%%%%%%%%%%%%%%%%%%%%%%%%%%%%%%%%%%%%%%%%%%%%%%%%%%%%%%%%%%%%%%%%%%%%%%%%%%%%%%%%%%%%%%%%%%%%%%%%%%%%%%%%%%%%%%%%%%%%%%%%
\end{document}